\documentclass[11pt]{article}
\usepackage[margin=1in]{geometry} 
\usepackage{amssymb,amsfonts,amsmath,bbm,mathrsfs,stmaryrd}
\usepackage{xcolor}
\usepackage{url}

\usepackage{enumerate}

\usepackage[colorlinks,
             linkcolor=black!75!red,
             citecolor=blue,
             pdftitle={},
             pdfproducer={pdfLaTeX},
             pdfpagemode=None,
             bookmarksopen=true
             bookmarksnumbered=true]{hyperref}
            
\usepackage{tikz}
\usetikzlibrary{arrows,calc,decorations.pathreplacing,decorations.markings,intersections,shapes.geometric,through,fit,shapes.symbols,positioning,decorations.pathmorphing}

\usepackage{braket}

\usepackage[amsmath,thmmarks,hyperref]{ntheorem}
\usepackage{cleveref}

\creflabelformat{enumi}{#2(#1)#3}

\crefname{section}{Section}{Sections}
\crefformat{section}{#2Section~#1#3} 
\Crefformat{section}{#2Section~#1#3} 

\crefname{subsection}{\S}{\S\S}
\AtBeginDocument{%
  \crefformat{subsection}{#2\S#1#3}%
  \Crefformat{subsection}{#2\S#1#3}%
}

\theoremstyle{plain}

\newtheorem{lemma}{Lemma}[section]
\newtheorem{proposition}[lemma]{Proposition}
\newtheorem{corollary}[lemma]{Corollary}
\newtheorem{theorem}[lemma]{Theorem}

\theoremstyle{nonumberplain}

\theoremstyle{plain}
\theorembodyfont{\upshape}
\theoremsymbol{\ensuremath{\blacklozenge}}

\newtheorem{definition}[lemma]{Definition}
\newtheorem{example}[lemma]{Example}
\newtheorem{remark}[lemma]{Remark}

\crefname{definition}{definition}{definitions}
\crefformat{definition}{#2definition~#1#3} 
\Crefformat{definition}{#2Definition~#1#3} 

\crefname{ex}{example}{examples}
\crefformat{example}{#2example~#1#3} 
\Crefformat{example}{#2Example~#1#3} 

\crefname{remark}{remark}{remarks}
\crefformat{remark}{#2remark~#1#3} 
\Crefformat{remark}{#2Remark~#1#3} 

\crefname{convention}{convention}{conventions}
\crefformat{convention}{#2convention~#1#3} 
\Crefformat{convention}{#2Convention~#1#3} 

\crefname{notation}{notation}{notations}
\crefformat{notation}{#2notation~#1#3} 
\Crefformat{notation}{#2Notation~#1#3} 

\crefname{question}{question}{questions}
\crefformat{question}{#2question~#1#3} 
\Crefformat{question}{#2Question~#1#3} 

\crefname{lemma}{lemma}{lemmas}
\crefformat{lemma}{#2lemma~#1#3} 
\Crefformat{lemma}{#2Lemma~#1#3} 

\crefname{proposition}{proposition}{propositions}
\crefformat{proposition}{#2proposition~#1#3} 
\Crefformat{proposition}{#2Proposition~#1#3} 

\crefname{corollary}{corollary}{corollaries}
\crefformat{corollary}{#2corollary~#1#3} 
\Crefformat{corollary}{#2Corollary~#1#3} 

\crefname{theorem}{theorem}{theorems}
\crefformat{theorem}{#2theorem~#1#3} 
\Crefformat{theorem}{#2Theorem~#1#3} 

\crefname{enumi}{}{}
\crefformat{enumi}{(#2#1#3)}
\Crefformat{enumi}{(#2#1#3)}

\crefname{assumption}{assumption}{Assumptions}
\crefformat{assumption}{#2assumption~#1#3} 
\Crefformat{assumption}{#2Assumption~#1#3} 

\crefname{equation}{}{}
\crefformat{equation}{(#2#1#3)} 
\Crefformat{equation}{(#2#1#3)}

\numberwithin{equation}{section}

\theoremstyle{nonumberplain}
\theoremsymbol{\ensuremath{\blacksquare}}

\newtheorem{proof}{Proof}
\newcommand\pf[1]{\newtheorem{#1}{Proof of \Cref{#1}}}

\newcommand\bC{{\mathbb C}}

\newcommand\bR{{\mathbb R}}

\newcommand\cM{{\mathcal M}}

\newcommand\cO{{\mathcal O}}

\newcommand\cU{{\mathcal U}}
\newcommand\cV{{\mathcal V}}

\DeclareMathOperator{\id}{id}

\newcommand{\qedhere}{\mbox{}\hfill\ensuremath{\blacksquare}}

\title{Quantum isometries and loose embeddings}
\author{Alexandru Chirvasitu}

\begin{document}

\date{}

\newcommand{\Addresses}{{
  \bigskip
  \footnotesize

  \textsc{Department of Mathematics, University at Buffalo, Buffalo,
    NY 14260-2900, USA}\par\nopagebreak \textit{E-mail address}:
  \texttt{achirvas@buffalo.edu}

}}

\maketitle

\begin{abstract}
  We show that countable metric spaces always have quantum isometry groups, thus extending the class of metric spaces known to possess such universal quantum-group actions.

  Motivated by this existence problem we define and study the notion of loose embeddability of a metric space $(X,d_X)$ into another, $(Y,d_Y)$: the existence of an injective continuous map that preserves both equalities and inequalities of distances. We show that $0$-dimensional compact metric spaces are ``generically'' loosely embeddable into the real line, even though not even all countable metric spaces are.
\end{abstract}


\vspace{.5cm}



\section*{Introduction}

Let $(X,d)$ be a compact metric space and 
\begin{equation}\label{eq:act}
  \alpha:X\times G\to X
\end{equation}
an isometric action of a compact group $G$ on $X$. Every metric space admits a {\it universal} $\alpha$, in the sense that every compact-group isometric action $X\times H\to X$ arises via a unique compact-group morphism $H\to G$. Indeed, one simply takes $G=\mathrm{Iso}(X,d)$ (the group of all self-isometries of $X$), equipped with the uniform topology.  

The present note is partly motivated by the question of whether such universal isometric actions on compact metric spaces $X$ exist in the context of compact {\it quantum} groups. To make sense of this one dualizes \Cref{eq:act} to a unital $C^*$-algebra morphism
\begin{equation*}
  \rho:C(X)\to C(X)\otimes Q
\end{equation*}
where $Q$ is a compact quantum group (see \Cref{subse.cqg} below for detailed definitions). We then have a concept of $\rho$ being {\it isometric} (\cite[Definition 3.1]{metric} and \Cref{def:qiso} below), and can similarly pose in fairly guessable manner the question of whether there is a ``largest'' isometric quantum action (\Cref{def:univ}). If there is, we refer to the quantum group in question as the {\it quantum isometry group} of $(X,d)$.

One of the main results of \cite{metric} (Theorem 4.8 therein) is that quantum isometry groups always exist for compact metric spaces isometrically embeddable in some finite-dimensional Hilbert space. More is true however: according to (a slightly paraphrased version of) \cite[Corollary 4.9]{metric}, we have

\begin{theorem}\label{th:orig}
  A compact metric space $(X,d)$ has a quantum isometry group provided it embeds into some finite-dimensional Hilbert space by a continuous one-to-one map that preserves equalities and differences of distances.
  \qedhere
\end{theorem}

The phrasing above is a bit awkward; what is meant is that we have a map $f:X\to \bR^n$ (for some $n$) such that for any four points $x,z,x',z'$ in $X$ the distances $d(x,y)$ and $d(x',y')$ are equal if and only if their Euclidean counterparts
\begin{equation*}
  |fx-fy|\text{ and }|fx'-fy'|
\end{equation*}
are equal. This motivates the natural question (now entirely separate of the issue of quantum actions) of which compact metric spaces admit such {\it loose embeddings} (a term we introduce in \Cref{def.lo-emb}) into Euclidean spaces. The term is meant to convey the fact that such an embedding demands much less that distance preservation; the defining condition \Cref{eq:le} is essentially combinatorial in nature, concerned, as it is, only with the pattern of equalities between pairwise distances in $(X,d)$. 

The main result of the short \Cref{se.qiso} is

\begin{theorem}[\Cref{th.cnt}]
 All countable compact metric spaces have quantum isometry groups
\end{theorem}

We then transition to the loose-embeddability material (\Cref{def.lo-emb}). Note however that by \Cref{ex.pth} that condition is not {\it necessary} for the existence of quantum isometry groups (i.e. the converse to \Cref{th:orig} is not valid). 

Although \Cref{ex.pth} shows that loose embeddability is not automatic even for countable metric spaces, \Cref{th.gen0} proves that ``most'' $0$-dimensional (i.e. totally disconnected) compact metric spaces are loosely embeddable. Formally, \Cref{th.gen0} (with a fragment of \Cref{pr.m0baire} thrown in for clarity) reads:

\begin{theorem}
  The set $\cM_{\le 0}$ of isometry classes of $0$-dimensional compact metric spaces is a dense $G_{\delta}$ in the complete Gromov-Hausdorff metric space $\cM$ of isometry classes of {\it all} compact metric spaces, and hence $\cM_{\le 0}$ is a Baire space. 
  
Furthermore, the complement in $\cM_{\le 0}$ of the set of isometry classes of loosely embeddable $0$-dimensional compact metric spaces is of first Baire category.   
  \qedhere
\end{theorem}

\subsection*{Acknowledgements}

This work was partially supported by NSF grants DMS-1801011 and DMS-2001128.

I am grateful for insightful comments by Hanfeng Li, Vern Paulsen and last but decidedly not least the anonymous referee.

\section{Preliminaries}\label{se.prel}

Unless specified otherwise, all algebras and morphisms between them are assumed unital. We will frequently have to take tensor products of $C^*$-algebras, in which case the tensor symbol always denotes the minimal (or {\it spatial}) tensor product \cite[Definition 3.3.4]{bo}. On the other hand, between plain, non-topological algebras `$\otimes$' denotes the usual, algebraic tensor product. 

\subsection{Metric geometry}\label{subse.mtr}

We gather some background material on metric geometry and point-set topology with \cite{bbi,mun,eng-dim} serving as references. Recall (e.g. \cite[Definitions 7.3.1 and 7.3.10]{bbi})

\begin{definition}\label{def.gh}
  Let $(Z,d)$ be a metric space and $X,Y\subseteq Z$ two subsets. The {\it Hausdorff distance} $d_{H,Z}(X,Y)$ is the infimum over all $\varepsilon>0$ such that $X$ and $Y$ are each contained in the other's $\varepsilon$-neighborhood.

  For two metric spaces $(X,d_X)$ and $(Y,d_Y)$ the {\it Gromov-Hausdorff distance} is defined as
  \begin{equation*}
    d_{GH}(X,Y) = \inf d_{H,Z}(X,Y)
  \end{equation*}
  where the infimum is taken over all metric spaces $Z$ housing $X$ and $Y$ as isometrically embedded subspaces.
\end{definition}

Gromov-Hausdorff distance is an actual metric on the set of isometry classes of compact metric spaces \cite[Theorem 7.3.30]{bbi}, and we write $(\cM,d_{GH})$ for the resulting metric space. By abuse of notation, we often identify a compact metric space with its corresponding point of $\cM$ (i.e. its isometry class).  

We will also need the notion of {\it dimension} for a topological space. As explained throughout \cite[Chapter 1]{eng-dim}, the various competing definitions do not, in general agree. They do, however, for compact metric spaces (or more generally, separable ones) \cite[Theorem 1.7,7]{eng-dim}. For that reason, we provide one of the definitions (of what is usually called the {\it covering dimension}; see \cite[Definition 1.6.7]{eng-dim} or \cite[\S 50]{mun}) and omit all qualifiers preceding the term `dimension'. 

\begin{definition}\label{def.dim}
  Let $X$ be a compact metrizable topological space.

  A finite open cover $\cU=\{U_i\}$ of $X$ has {\it order $n\ge 0$} if there are $n+1$ mutually intersecting sets $U_i$ but no $n+2$ sets $U_i$ intersect. The order is infinite if no such $n$ exists.

$X$ is said to have {\it dimension} $\dim X=n$ if every finite open cover has a finite refinement of order $\le n$, and $n$ is the smallest integer with this property ($\dim X=\infty$ if no such integers exist). By convention, $\dim \emptyset=-1$. 
\end{definition}

A few remarks worth keeping in mind:
\begin{itemize}
\item for manifolds $\dim X$ coincides with the standard concept.
\item for compact metric spaces $0$-dimensionality means total disconnectedness, i.e. the existence of a basis consisting of clopen sets \cite[Theorem 1.4.5]{eng-dim}.
\item for separable metric spaces dimension is (as expected) monotonic with respect to inclusions \cite[Theorem 3.1.19]{eng-dim}.
\end{itemize}

We indicate dimension constraints for elements of $\cM$ by a subscript: $\cM_0\subset \cM$ for instance denotes the set (of isometry classes) of $0$-dimensional compact metric spaces, $\cM_{\le n}$ that of metric spaces of dimension $\le n$, etc.

$(\cM,d_{GH})$ is known to be a {\it complete} (separable) metric space and hence, by the Baire category theorem (\cite[Theorem 48.2]{mun}), a {\it Baire space} in the sense of \cite[\S 48]{mun}: countable intersections of dense open subsets are again dense. 

Baire's theorem suggests that sets containing countable intersections of dense open sets should be regarded as ``large''. We recall the relevant language:

\begin{definition}\label{def.baire}
  Let $X$ be a topological space. A subset $Y\subseteq X$ is {\it meager} or {\it of first category} if it is contained in a countable union of nowhere dense closed subsets of $X$.

  A subset that is not meager is {\it of second category}, and the complement of a meager set is {\it residual}.
\end{definition}

\begin{remark}
  In Baire spaces being residual is equivalent to containing a dense {\it $G_{\delta}$ set}, i.e. countable intersection of open subsets. 
\end{remark}

The Baire theorem applies not only to complete metric spaces but to $G_{\delta}$ subsets of the latter (\cite[\S 48, Exercise 5]{mun}). This makes the following result relevant.

\begin{proposition}\label{pr.m0baire}
  For every non-negative integer $n$, the subspace $\cM_{\le n}\subset \cM$ consisting of (isometry classes of) compact metric spaces of dimension $\le n$ is dense $G_{\delta}$, and hence a Baire space.
\end{proposition}
\begin{proof}
  The density claim follows from the fact that the set of finite metric spaces (and hence also $\cM_{\le 0}$) is dense in $\cM$.
  
  Fix positive integers $M$, $N$ and let $\cM^{N}_{M}\subset \cM$ be the set of metric spaces $(X,d)$ admitting some finite open cover $\cU=\{U_i\}$ such that
  \begin{itemize}
  \item $\cU$ has {\it mesh} $<\frac 1N$, i.e. the supremum of the diameters $\mathrm{diam}\; U_i$ is $<\frac 1N$;
  \item for each $n+2$-element subset $\cV\subset \cU$ and each $U\in \cV$ we have
    \begin{equation*}
      \inf_{x\in U,y\in V}d(x,y)>\frac 1M
    \end{equation*}
    where
    \begin{equation*}
      V=\bigcap_{U'\in \cV,\ U'\ne U}U'.
    \end{equation*}
  \end{itemize}

  By the characterization of dimension for compact metric spaces given in \cite[Theorem 1.6.12]{eng-dim} we have
  \begin{equation*}
    \cM_{\le n} = \bigcap_{N\to \infty}\bigcup_{M\to \infty}\cM^{N}_{M}.
  \end{equation*}
  Since the intersection can be indexed by reciprocals $\varepsilon=\frac 1N$ of positive integers as $N\to \infty$, the conclusion will follow once we prove

  {\bf Claim: $\cM^{N}_{M}$ is open in $\cM$.} To verify this, let $(X,d_X)\in \cM^N_M$ and consider a cover $\cU=\{U_i\}$ as in the definition of the latter. Let $\delta>0$ (more on its size later) and suppose the elements
  \begin{equation*}
    (X,d_X),(Y,d_Y)\in \cM
  \end{equation*}
  are isometrically embedded in a compact metric space $(Z,d_Z)$ and $\delta$-Hausdorff close therein. The open subsets
  \begin{equation*}
    U_i^{\delta}:=\{z\in Z\ |\ d_Z(z,U_i)<\delta\}\subseteq Z
  \end{equation*}
  will then cover $Y\subseteq Z$. 

  If $\delta=\delta(X,d_X)$ is sufficiently small then we can ensure that every intersection
  \begin{equation*}
    U^{\delta}_{i_0}\cap\cdots\cap U^{\delta}_{i_t}\subset Z,\ 1\le t\le n
  \end{equation*}
  of at most $n+1$ open sets is contained in the $\delta'$-neighborhood of the corresponding intersection 
  \begin{equation*}
    U_{i_0}\cap\cdots\cap U_{i_t}\subset X
  \end{equation*}
  for arbitrarily small $\delta'>0$ (that would have to be fixed {\it before} $\delta$, $Y$, $Z$, etc.).

  In turn, requiring $\delta'>0$ sufficiently small would ensure that the open cover of $Y$ by $V_i:=U_i^{\delta}\cap Y$ satisfies the two requirements in the definition of $\cM^N_M$ and hence witnesses $Y$'s membership in that subset of $\cM$.
\end{proof}

In particular:

\begin{corollary}
  The subspace $\cM_0\subset \cM$ is $G_{\delta}$ and hence Baire. 
  \qedhere
\end{corollary}

\subsection{Compact quantum groups and actions}\label{subse.cqg}

For the material in this subsection we refer for instance to \cite{Van,Pseudogroup,Wor98,KusTus99}, recalling only skeletal background here. We will also need the very basics of Hopf algebra theory, for which \cite{swe,mntg,abe,rad} constitute good references. 

\begin{definition}\label{def:cqg}
  A {\it compact quantum group} is a unital $C^*$-algebra $Q$ equipped with a $C^*$ morphism $\Delta:Q\to Q\otimes Q$ which
  \begin{itemize}
  \item is {\it coassociative} in the sense that
    \begin{equation*}
 \begin{tikzpicture}[auto,baseline=(current  bounding  box.center)]
  \path[anchor=base] (0,0) node (1) {$Q$} 
  +(5,0) node (2) {$Q\otimes Q\otimes Q$}
  +(2,.5) node (u) {$Q\otimes Q$}
  +(2,-.5) node (d) {$Q\otimes Q$}
  ;
  \draw[->] (1) to[bend left=6] node[pos=.5,auto] {$\scriptstyle \Delta$} (u);
  \draw[->] (1) to[bend right=6] node[pos=.5,auto,swap] {$\scriptstyle \Delta$} (d);
  \draw[->] (u) to[bend left=6] node[pos=.5,auto] {$\scriptstyle \Delta\otimes \id$} (2);
  \draw[->] (d) to[bend right=6] node[pos=.5,auto,swap] {$\scriptstyle \id\otimes \Delta$} (2);
 \end{tikzpicture}
\end{equation*}
commutes;
\item the spans
  \begin{equation*}
    \Delta(Q)(\bC\otimes Q)\quad\text{and}\quad \Delta(Q)(Q\otimes \bC)
  \end{equation*}
  of the respective products are dense in $Q\otimes Q$. 
  \end{itemize}
\end{definition}

In general, we regard unital $C^*$-algebras as objects dual to {\it compact quantum spaces}; this terminology will be in use throughout. For that reason, we will write $Q=C(G)$ (and refer to $G$ as the compact quantum group) to emphasize that the compact quantum group $G$ is to be regarded as dual to its algebra $Q$ of continuous functions.

A compact quantum group $C(G)$ has a unique dense $*$-subalgebra
\begin{equation*}
  \cO(G)\subseteq C(G)
\end{equation*}
that becomes a Hopf $*$-algebra when equipped with the comultiplication $\Delta$ inherited from $C(G)$; in particular, this means that
\begin{equation*}
  \Delta(\cO(G))\subset \cO(G)\otimes \cO(G)\subset C(G)\otimes C(G);
\end{equation*}
see for instance \cite[Theorem 3.1.7]{KusTus99}. The antipode $\kappa$ of $\cO(G)$ need not extend continuously to $C(G)$. 

$C(G)$ has a {\it Haar state} $h:C(G)\to \bC$, left and right $G$-invariant (as the Haar measure is on classical compact groups) in the sense that
\begin{equation*}
  h*\varphi = h = \varphi*h,\ \forall \text{ states }\varphi\text{ on }C(G)
\end{equation*}
where
\begin{equation*}
 \begin{tikzpicture}[auto,baseline=(current  bounding  box.center)]
  \path[anchor=base] (0,0) node (1) {$C(G)$} 
  +(6,0) node (2) {$\bC$}
  +(3,1) node (u) {$C(G)\otimes C(G)$}
  ;
  \draw[->] (1) to[bend right=6] node[pos=.5,auto,swap] {$\scriptstyle \varphi*\psi$} (2);
  \draw[->] (1) to[bend left=6] node[pos=.5,auto] {$\scriptstyle \Delta$} (u);
  \draw[->] (u) to[bend left=6] node[pos=.5,auto] {$\scriptstyle \varphi\otimes \psi$} (2);
 \end{tikzpicture}
\end{equation*}
defines the convolution product on functionals on $C(G)$.

\begin{definition}
  Let $G$ be a compact quantum group.
  
  $C(G)$ is {\it reduced} if the Haar state $h:C(G)\to \bC$ is faithful.

  $C(G)$ is {\it full} if the map $\cO(G)\to C(G)$ is the $C^*$ envelope of the complex $*$-algebra $\cO(G)$.

  For arbitrary $G$ we write $C(G)_r$ for the {\it reduced version} of $G$, i.e. the image of the GNS representation of the Haar state $h:C(G)\to \bC$, and $C(G)_u$ for the {\it full (or universal) version} of $G$, i.e. the $C^*$ envelope of $\cO(G)$ (such an envelope exists for every $G$).
\end{definition}

Note that we have quantum group morphisms
\begin{equation}\label{eq:sandw}
  C(G)_u\to C(G)\to C(G)_r. 
\end{equation}

\begin{definition}\label{def:act}
  An {\it action} of a compact quantum group $(Q,\Delta)$ on a compact quantum space $A$ is a $C^*$-morphism
  \begin{equation}\label{eq:rho}
    \rho:A\to A\otimes Q
  \end{equation}
  such that
  \begin{itemize}
  \item     \begin{equation*}
      \begin{tikzpicture}[auto,baseline=(current  bounding  box.center)]
        \path[anchor=base] (0,0) node (1) {$A$} 
        +(5,0) node (2) {$A\otimes Q\otimes Q$}
        +(2,.5) node (u) {$A\otimes Q$}
        +(2,-.5) node (d) {$A\otimes Q$}
        ;
        \draw[->] (1) to[bend left=6] node[pos=.5,auto] {$\scriptstyle \rho$} (u);
        \draw[->] (1) to[bend right=6] node[pos=.5,auto,swap] {$\scriptstyle \rho$} (d);
        \draw[->] (u) to[bend left=6] node[pos=.5,auto] {$\scriptstyle \rho\otimes \id$} (2);
        \draw[->] (d) to[bend right=6] node[pos=.5,auto,swap] {$\scriptstyle \id\otimes \Delta$} (2);
      \end{tikzpicture}
    \end{equation*}
    commutes;
  \item the span $\rho(A)(\bC\otimes Q)$ is dense in $A\otimes Q$.
  \end{itemize}

  If furthermore the span of 
  \begin{equation*}
    \{(\varphi\otimes \id)\rho(a)\ |\ a\in A,\ \varphi\text{ a state on }Q\}\subset Q
  \end{equation*}
  is dense then the action is {\it faithful}. 
\end{definition}

Composing
\begin{equation*}
 \begin{tikzpicture}[auto,baseline=(current  bounding  box.center)]
   \path[anchor=base] (0,0) node (1) {$A$} 
   +(3,0) node (2) {$A\otimes C(G)$}
   +(6,0) node (3) {$A\otimes C(G)_r$}
   ;
   \draw[->] (1) to[bend left=0] node[pos=.5,auto] {$\scriptstyle \rho$} (2);
   \draw[->] (2) to[bend left=0] node[pos=.5,auto] {$\scriptstyle \id\otimes \pi$} (3);
 \end{tikzpicture}
\end{equation*}
with $\pi:C(G)\to C(G)_r$ the canonical surjection from \Cref{eq:sandw} produces an action of the reduced version $C(G)_r$, so if needed we can always assume that an acting quantum group is reduced.

Throughout the present paper we in fact only work with {\it classical} spaces $A$, i.e. $A=C(X)$ (continuous functions) for some compact Hausdorff $X$. It can be shown \cite[Theorem 3.16]{Hua12} that for $G$ acting faithfully on $X$ the antipode $\kappa$ of the Hopf $*$-algebra $\cO(G)$ extends continuously to $C(G)_r$ (so $G$ is {\it of Kac type}, in standard terminology). For that reason, we will henceforth assume all of our compact quantum groups $C(G)$ come equipped with antipodes $\kappa$.

We are interested primarily in compact metric spaces $(X,d)$. In that context, the relevant notion of structure-preserving quantum-group action was introduced in \cite[Definition 3.1]{metric}:

\begin{definition}\label{def:qiso}
  Let $(X,d)$ be a compact metric space, $A=C(X)$ and $\rho:A\to A\otimes Q$ a compact-quantum-group action on $X$. $\rho$ is {\it isometric} if we have
  \begin{equation*}
    \rho(d_x)(y) = \kappa(\rho(d_y)(x)) \in Q
  \end{equation*}
   for all pairs of points $x,y\in X$, where $\kappa:Q\to Q$ is the antipode. 
\end{definition}

Finally, we can give

\begin{definition}\label{def:univ}
  Let $(X,d)$ be a compact metric space and $A=C(X)$. An isometric action \Cref{eq:rho} is {\it universal} if every isometric action $\rho':A\to A\otimes H$ factors as
\begin{equation*}
 \begin{tikzpicture}[auto,baseline=(current  bounding  box.center)]
  \path[anchor=base] (0,0) node (1) {$A$} 
  +(4,0) node (2) {$A\otimes Q$}
  +(2,1) node (u) {$A\otimes H$}
  ;
  \draw[->] (1) to[bend right=6] node[pos=.5,auto,swap] {$\scriptstyle \varphi*\psi$} (2);
  \draw[->] (1) to[bend left=6] node[pos=.5,auto] {$\scriptstyle \id\otimes \eta$} (u);
  \draw[->] (u) to[bend left=6] node[pos=.5,auto] {$\scriptstyle \rho'$} (2);
 \end{tikzpicture}
\end{equation*}
for a unique compact quantum group morphism $\eta:Q\to H$.

If such a universal action exists we say that $(X,d)$ {\it has a quantum isometry group}. 
\end{definition}

\section{Countable metric spaces}\label{se.qiso}

\begin{theorem}\label{th.cnt}
  A countable compact metric space $(X,d)$ has a compact quantum isometry group.   
\end{theorem}
\begin{proof}
Consider an isometric action of a CQG $G$ on $(X,d)$. 
  
  Being countable, $X$ must contain isolated points. Each isolated point, in turn, is contained in one of the finite sets
  \begin{equation*}
    X_{\ge r}:=\{x\in X\ |\ d(x,y)\ge r,\ \forall x\ne y\in X\} 
  \end{equation*}
  for some $r>0$ (this is the set of points which admit no neighbors at a distance smaller than $r$). \cite[Theorem 3.1]{Chi15}, for instance, makes it clear that each $X_{\ge r}$, $r\ge 0$ is preserved by the action of $G$. 

  Letting $r\to 0$, we see that the action of $G$ leaves invariant the entire set
  \begin{equation*}
    X_{\mathrm{iso}}=\bigcup_{r\to 0}X_{\ge r}
  \end{equation*}
  of isolated points. That set is open because each $X_{\ge r}$ is, so the compact countable metric space $X\setminus X_{\mathrm{iso}}$ is again preserved. Now repeat the procedure with the latter space eliminating {\it its} isolated points, and so on. This transfinite recursive procedure, which must terminate after countably many steps, will partition the original metric space $X$ into countably many {\it finite} subspaces preserved by $G$. Since these spaces do not depend on $G$ or the action but are rather intrinsic to $X$, every isometric action will preserve them. 

  To conclude, denote by
  \begin{equation*}
    X_0\subset X_1\subset\cdots\subset X_{\alpha} = X
  \end{equation*}
  be the chain of invariant subspaces constructed above, where
  \begin{itemize}
  \item $\alpha$ is some countable ordinal;
  \item for limit ordinals $\beta$ the space $X_{\beta}$ is the closure of $\cup_{\beta'<\beta}X_{\beta'}$;
  \item for successor ordinals $X_{\beta+1}$ the space
    \begin{equation}\label{eq:findif}
      X_{\beta+1}\setminus X_{\beta}
    \end{equation}
    is finite. 
  \end{itemize}
  We can find a probability measure $\mu$ on $X$ by collecting measures on the finite spaces \Cref{eq:findif} invariant under the quantum isometry groups of the latter, whereupon $\mu$ will be invariant under any isometric action on $(X,d)$ of a compact quantum group. Since we know (e.g. by \cite[Theorem 5.4]{cg}) that there is a universal compact quantum group acting isometrically on $(X,d)$ {\it and} preserving $\mu$, this proves the desired conclusion.
\end{proof}

We will need the following notion.

\begin{definition}\label{def.lo-emb}
  Let $(X,d_X)$ and $(Y,d_Y)$ be two metric spaces. A {\it loose} (or {\it loosely isometric}) {\it embedding} $X\to Y$ is a one-to-one continuous map $f:X\to Y$ with the property that for every $x,x',z$ and $z'$ in $X$ we have
  \begin{equation}\label{eq:le}
    d_Y(fx,fx') = d_Y(fz,fz') \Leftrightarrow d_X(x,x') = d_X(z,z').
  \end{equation}

  We say that $(X,d)$ is {\it loosely embeddable} (or LE for short) if there is a loose embedding into some Euclidean space $\bR^n$ with its usual distance function, typically denoted by
  \begin{equation*}
    |x-y| := d_{\bR^n}(x,y) = \left(\sum_{i=1}^n (x_i-y_i)^2\right)^{\frac 12}. 
  \end{equation*}
\end{definition}

A slight generalization of \cite[Corollary 4.9]{metric} says that loosely embeddable compact metric spaces have compact quantum automorphism groups. \Cref{ex.pth} shows that not every countable compact metric space is loosely embeddable, and hence not all spaces covered by \Cref{th.cnt} fall within the scope of that result.


\begin{example}\label{ex.pth}
  One can easily construct countable compact metric spaces $(X,d)$ such that for every $n$, $X$ contains {\it regular $n$-simplices}, i.e. $(n+1)$-tuples of equidistant points $x_0$, $x_1$ up to $x_n$. Such a space cannot admit a loose embedding into any Euclidean space $\bR^d$, since the latter cannot house a regular simplex with more than $d+1$ vertices.
\end{example}

\begin{remark}
  Contrast \Cref{ex.pth} to the fact that by \cite[Corollary 3]{dm} {\it finite} metric spaces are always loosely embeddable. 
\end{remark}


\section{Loose metric embeddability}\label{se:loose}

It is a natural problem, in view of \Cref{ex.pth} and the discussion preceding it, to determine to what extent various classes of metric spaces are loosely embeddable. Clearly, loose embeddability of a compact metric space in $\bR^n$ entails covering dimension $\le n$ (e.g. \cite[\S 1.6]{eng-dim} and \cite[Theorem 3.1.19]{eng-dim}, which applies to compact metric spaces). On the other hand, we will prove that if the dimension is zero then loose embeddability holds ``generically'' with respect to the Gromov-Hausdorff distance.

\begin{theorem}\label{th.gen0}
  The isometry classes of $0$-dimensional compact metric spaces loosely embeddable in $\bR$ is a residual set in $(\cM_0,d_{GH})$. 
\end{theorem}

We need some preparation.

\begin{definition}\label{def.inj}
  A distance function $d$ on a metric space $X$ is {\it injective} if its restriction to the off-diagonal set
  \begin{equation*}
    X\times X\setminus \Delta=\{(x,y)\in X\times X\ |\ x\ne y\}
  \end{equation*}
  is one-to-one.   
\end{definition}

First, note the following sufficient criterion for loose embeddability in $\bR$.

\begin{proposition}\label{pr.inj-suff}
  A $0$-dimensional compact metric space $(X,d)\in \cM_0$ with injective $d$ is loosely embeddable in $\bR$.
\end{proposition}
\begin{proof}
  The definition of loose embeddability simply requires a homeomorphism of $X$ onto a subset $Y\subset \bR$ so that the usual real line distance $d_{\bR}$ on $Y$ is injective. This will be a familiar clopen cover recursive ``branching'' procedure familiar in working with $0$-dimensional compact spaces:

In first instance, cover $X$ with disjoint clopen sets $U_i$ of diameter $\le 1$ and match them to mutually disjoint compact intervals $I_i\subset \bR$ of length $\le 1$. We arrange furthermore that the intervals $I_i$ are chosen {\it generically}, in the sense that if
\begin{align*}
  x\in I_i,\ &y\in I_j,\ i\ne j\\
  x'\in I_{i'},\ &y'\in I_{j'},\ i'\ne j'\\
\end{align*}
then
\begin{equation*}
  d_{\bR}(x,y)=d_{\bR}(x',y')\Rightarrow i=i'\text{ and }j=j'.
\end{equation*}
Next, cover each $U_i$ with finitely many disjoint clopen subsets $U_{ij}$ of diameter $\frac 12$ and choose corresponding disjoint compact sub-intervals $I_{ij}\subset I_i$ of length $\le \frac 12$, again ensuring that for each $i$ the family consisting of all $I_{ij}$ is generic in the above sense. Now continue the procedure, partitioning each $U_{ij}$ into clopen subsets $U_{ijk}$, etc.

For each infinite word $ijk\cdots$ we obtain a point
\begin{equation*}
  \{p\}=U_i\cap U_{ij}\cap U_{ijk}\cap\cdots\subset X
\end{equation*}
mapped by our embedding to the corresponding unique point
\begin{equation*}
  I_i\cap I_{ij}\cap I_{ijk}\cap\cdots.
\end{equation*}
The generic condition imposed on our intervals at each step then ensures that indeed the restriction of $d_{\bR}$ to the image of $X$ is injective in the sense of \Cref{def.inj}.
\end{proof}

\pf{th.gen0}
\begin{th.gen0}
In view of \Cref{pr.inj-suff}, it will suffice to prove the stronger claim that the collection of $(X,d)\in \cM_0$ with injective $d$ is residual. 

Let $M$ and $N$ be two positive integers and define $\cM_{N,M}\subset \cM_0$ be the collection of $0$-dimensional metric spaces $(X,d)$ admitting a partition into clopen subsets $U_i$ such that
\begin{itemize}
\item each $U_i$ has diameter $<\frac 1N$;
\item whenever the two-element sets $\{i,j\}$ and $\{i',j'\}$ are distinct we have
  \begin{equation*}
    |d(x,y)-d(x',y')|>\frac 1M
  \end{equation*}
  for all $x\in U_i$, $y\in U_j$ and similarly for primed symbols. 
\end{itemize}

$\cM_{N,M}$ is easily seen to be open in the Gromov-Hausdorff distance and the subset of $\cM_0$ consisting of injective-distance metric spaces is $\bigcap_N\bigcup_M \cM_{N,M}$. To conclude, we have to prove

{\bf Claim: $\bigcup_M\cM_{N,M}$ is dense in $\cM_0$.} This, however, is immediate: simply approximate an arbitrary compact metric space in the Gromov-Hausdorff topology by {\it finite} metric spaces, which can be chosen to have injective distance functions by effecting small perturbations on said distance functions if needed. 
\end{th.gen0}



\begin{thebibliography}{10}

\bibitem{abe}
Eiichi Abe.
\newblock {\em Hopf algebras}, volume~74 of {\em Cambridge Tracts in
  Mathematics}.
\newblock Cambridge University Press, Cambridge-New York, 1980.
\newblock Translated from the Japanese by Hisae Kinoshita and Hiroko Tanaka.

\bibitem{bo}
Nathanial~P. Brown and Narutaka Ozawa.
\newblock {\em {$C^*$}-algebras and finite-dimensional approximations},
  volume~88 of {\em Graduate Studies in Mathematics}.
\newblock American Mathematical Society, Providence, RI, 2008.

\bibitem{bbi}
Dmitri Burago, Yuri Burago, and Sergei Ivanov.
\newblock {\em A course in metric geometry}, volume~33 of {\em Graduate Studies
  in Mathematics}.
\newblock American Mathematical Society, Providence, RI, 2001.

\bibitem{Chi15}
Alexandru Chirvasitu.
\newblock On quantum symmetries of compact metric spaces.
\newblock {\em J. Geom. Phys.}, 94:141--157, 2015.

\bibitem{cg}
Alexandru Chirvasitu and Debashish Goswami.
\newblock Existence and {R}igidity of {Q}uantum {I}sometry {G}roups for
  {C}ompact {M}etric {S}paces.
\newblock {\em Comm. Math. Phys.}, 380(2):723--754, 2020.

\bibitem{dm}
M.~Deza and H.~Maehara.
\newblock Metric transforms and {E}uclidean embeddings.
\newblock {\em Trans. Amer. Math. Soc.}, 317(2):661--671, 1990.

\bibitem{eng-dim}
Ryszard Engelking.
\newblock {\em Dimension theory}.
\newblock North-Holland Publishing Co., Amsterdam-Oxford-New York; PWN---Polish
  Scientific Publishers, Warsaw, 1978.
\newblock Translated from the Polish and revised by the author, North-Holland
  Mathematical Library, 19.

\bibitem{metric}
Debashish Goswami.
\newblock Existence and examples of quantum isometry groups for a class of
  compact metric spaces.
\newblock {\em Adv. Math.}, 280:340--359, 2015.

\bibitem{Hua12}
Huichi Huang.
\newblock Invariant subsets under compact quantum group actions.
\newblock {\em J. Noncommut. Geom.}, 10(2):447--469, 2016.

\bibitem{KusTus99}
Johan Kustermans and Lars Tuset.
\newblock A survey of {$C^*$}-algebraic quantum groups. {I}.
\newblock {\em Irish Math. Soc. Bull.}, (43):8--63, 1999.

\bibitem{Van}
Ann Maes and Alfons Van~Daele.
\newblock Notes on compact quantum groups.
\newblock {\em Nieuw Arch. Wisk. (4)}, 16(1-2):73--112, 1998.

\bibitem{mntg}
Susan Montgomery.
\newblock {\em Hopf algebras and their actions on rings}, volume~82 of {\em
  CBMS Regional Conference Series in Mathematics}.
\newblock Published for the Conference Board of the Mathematical Sciences,
  Washington, DC; by the American Mathematical Society, Providence, RI, 1993.

\bibitem{mun}
James~R. Munkres.
\newblock {\em Topology}.
\newblock Prentice Hall, Inc., Upper Saddle River, NJ, 2000.
\newblock Second edition of [ MR0464128].

\bibitem{rad}
David~E. Radford.
\newblock {\em Hopf algebras}, volume~49 of {\em Series on Knots and
  Everything}.
\newblock World Scientific Publishing Co. Pte. Ltd., Hackensack, NJ, 2012.

\bibitem{swe}
Moss~E. Sweedler.
\newblock {\em Hopf algebras}.
\newblock Mathematics Lecture Note Series. W. A. Benjamin, Inc., New York,
  1969.

\bibitem{Pseudogroup}
S.~L. Woronowicz.
\newblock Compact matrix pseudogroups.
\newblock {\em Comm. Math. Phys.}, 111(4):613--665, 1987.

\bibitem{Wor98}
S.~L. Woronowicz.
\newblock Compact quantum groups.
\newblock In {\em Sym\'etries quantiques ({L}es {H}ouches, 1995)}, pages
  845--884. North-Holland, Amsterdam, 1998.

\end{thebibliography}
\def\polhk#1{\setbox0=\hbox{#1}{\ooalign{\hidewidth
  \lower1.5ex\hbox{`}\hidewidth\crcr\unhbox0}}}

\addcontentsline{toc}{section}{References}

\Addresses

\end{document}